\documentclass[12pt,a4paper,reqno]{amsart}
\usepackage{amssymb}
\usepackage{amscd}
\usepackage{enumerate}
\usepackage{graphicx}
\usepackage{siunitx}
\usepackage{tikz-cd}
\usepackage{color}
\usetikzlibrary{arrows}
\numberwithin{equation}{section}

\usepackage{mathabx}

\usepackage{mathtools}
\usepackage[tableposition=top]{caption}
\usepackage{booktabs,dcolumn}

\DeclareFontFamily{OT1}{rsfs}{}
\DeclareFontShape{OT1}{rsfs}{n}{it}{<-> rsfs10}{}
\DeclareMathAlphabet{\mathscr}{OT1}{rsfs}{n}{it}

\theoremstyle{plain}

\newtheorem{theorem}{Theorem}[section]
\newtheorem{proposition}[theorem]{Proposition}
\newtheorem{lemma}[theorem]{Lemma}

\newtheorem{conjecture}[theorem]{Conjecture}

\theoremstyle{definition}

\newtheorem{definition}[theorem]{Definition}
\newtheorem{remark}[theorem]{Remark}

\newcommand\R{\mathbb{R}}
\newcommand\Z{\mathbb{Z}}

\newcommand\C{\mathbb{C}}

\newcommand\eps{\varepsilon}
\newcommand\M{\operatorname{M}}
\newcommand\Grow{{\mathcal F}}
\newcommand\md{{\ \operatorname{mod}\ }}

\newcommand\id{{\operatorname{id}}}
\newcommand\nil{{\operatorname{nil}}}
\newcommand\sml{{\operatorname{sml}}}
\newcommand\unf{{\operatorname{unf}}}

\usepackage{algorithm, algorithmic}

\begin{document}

\title{Cancellation for the multilinear Hilbert transform}

\author{Terence Tao}
\address{UCLA Department of Mathematics, Los Angeles, CA 90095-1555.}
\email{tao@math.ucla.edu}



\begin{abstract}  For any natural number $k$, consider the $k$-linear Hilbert transform
$$ H_k( f_1,\dots,f_k )(x) := \operatorname{p.v.} \int_\R f_1(x+t) \dots f_k(x+kt)\ \frac{dt}{t}$$
for test functions $f_1,\dots,f_k: \R \to \C$.  It is conjectured that $H_k$ maps $L^{p_1}(\R) \times \dots \times L^{p_k}(\R) \to L^p(\R)$ whenever $1 < p_1,\dots,p_k,p < \infty$ and $\frac{1}{p} = \frac{1}{p_1} + \dots + \frac{1}{p_k}$. This is proven for $k=1,2$, but remains open for larger $k$.

In this paper, we consider the truncated operators
$$ H_{k,r,R}( f_1,\dots,f_k )(x) := \int_{r \leq |t| \leq R} f_1(x+t) \dots f_k(x+kt)\ \frac{dt}{t}$$
for $R > r > 0$. The above conjecture is equivalent to the uniform boundedness of $\| H_{k,r,R} \|_{L^{p_1}(\R) \times \dots \times L^{p_k}(\R) \to L^p(\R)}$ in $r,R$, whereas the Minkowski and H\"older inequalities give the trivial upper bound of $2 \log \frac{R}{r}$ for this quantity.  By using the arithmetic regularity and counting lemmas of Green and the author, we improve the trivial upper bound on $\| H_{k,r,R} \|_{L^{p_1}(\R) \times \dots \times L^{p_k}(\R) \to L^p(\R)}$ slightly to $o( \log \frac{R}{r} )$ in the limit $\frac{R}{r} \to \infty$ for any admissible choice of $k$ and $p_1,\dots,p_k,p$.  This establishes some cancellation in the $k$-linear Hilbert transform $H_k$, but not enough to establish its boundedness in $L^p$ spaces.
\end{abstract}

\maketitle


\section{Introduction}

For any natural number $k$ and test functions $f_1,\dots,f_k: \R \to \C$, define the $k$-linear Hilbert transform $H_k(f_1,\dots,f_k): \R \to \C$ by the formula
$$ H_k( f_1,\dots,f_k )(x) := \operatorname{p.v.} \int_\R f_1(x+t) \dots f_k(x+kt)\ \frac{dt}{t},$$
or more explicitly
\begin{equation}\label{hkdef}
 H_k( f_1,\dots,f_k )(x) = \lim_{r \to 0, R \to \infty} H_{k,r,R}(f_1,\dots,f_k)(x)
\end{equation}
where $H_{k,r,R}$ is the truncated $k$-linear Hilbert transform
\begin{equation}\label{hkr}
 H_{k,r,R}( f_1,\dots,f_k )(x) := \int_{r \leq |t| \leq R} f_1(x+t) \dots f_k(x+kt)\ \frac{dt}{t}.
\end{equation}
The operator $H_1$ is the classical Hilbert transform, which as is well known (see e.g. \cite{stein}) is bounded on $L^p(\R)$ for every $1 < p < \infty$.  The operator $H_2$ is the bilinear Hilbert transform; it was shown by Lacey and Thiele \cite{lacey-1, lacey-2} using time-frequency analysis techniques that $H_2$ maps $L^{p_1}(\R) \times L^{p_2}(\R)$ to $L^p(\R)$ whenever $1 < p,p_1,p_2 < \infty$ and $\frac{1}{p_1} + \frac{1}{p_2} = \frac{1}{p}$; in fact, they were able to relax the constraint $p>1$ to $p>2/3$, however in this paper it will be convenient to restrict\footnote{See also the negative results of \cite{demeter} showing that $H_3$ can be unbounded for certain values of $p$ below $1$.} to the ``Banach space case'' when all exponents are greater than $1$.  The same argument shows the corresponding bounds for $H_{2,r,R}$ that are uniform in $r,R$; that is to say, one has
$$ \| H_{2,r,R}( f_1, f_2 ) \|_{L^p(\R)} \leq C_{p,p_1,p_2} \|f_1\|_{L^{p_1}(\R)} \|f_2\|_{L^{p_2}(\R)}$$
whenever $1 < p,p_1,p_2 < \infty$, $\frac{1}{p_1} + \frac{1}{p_2} = \frac{1}{p}$, $f_1 \in L^{p_1}(\R)$, $f_2 \in L^{p_2}(\R)$, and $0 < r < R$, where $C_{p,p_1,p_2}$ is a quantity independent of $r,R$.  Note that the condition $\frac{1}{p_1} + \frac{1}{p_2} = \frac{1}{p}$ is necessary from dimensional analysis (or scaling) considerations.

From these facts, one may make the following conjecture.

\begin{conjecture}\label{tri}  Let $k \geq 1$ and $1 < p_1,\dots,p_k,p < \infty$ be such that $\frac{1}{p} = \frac{1}{p_1} + \dots + \frac{1}{p_k}$.  Then one has
\begin{equation}\label{hkp}
 \| H_{k,r,R}( f_1, \dots, f_k ) \|_{L^p(\R)} \leq C_{k,p,p_1,\dots,p_k} \|f_1\|_{L^{p_1}(\R)} \dots \|f_k\|_{L^{p_k}(\R)}
\end{equation}
whenever $1 < p,p_1,\dots,p_k < \infty$, $\frac{1}{p_1} + \dots + \frac{1}{p_k} = \frac{1}{p}$, $f_i \in L^{p_i}(\R)$ for $i=1,\dots,k$, and $0 < r < R$, where $C_{k,p,p_1,\dots,p_k}$ is a quantity independent of $r,R$.  In particular, from \eqref{hkdef} and Fatou's lemma we have
\begin{equation}\label{hkpr}
 \| H_{k}( f_1, \dots, f_k ) \|_{L^p(\R)} \leq C_{k,p,p_1,\dots,p_k} \|f_1\|_{L^{p_1}(\R)} \dots \|f_k\|_{L^{p_k}(\R)}
\end{equation}
for all test functions $f_1,\dots,f_k: \R \to \C$.
\end{conjecture}

As mentioned above, this conjecture is established for $k=1,2$, but is completely open for larger values of $k$.  For instance, in the case $k=3$ of the trilinear Hilbert transform $H_3$, no $L^p$ bounds whatsoever are known.  Although it is not needed to motivate our main results, we also remark that the implication of \eqref{hkpr} from \eqref{hkp} can be reversed (with some loss in the multiplicative constant); if \eqref{hkpr} holds, then by restricting $f_1,\dots,f_k$ to intervals of length $R$, applying \eqref{hkpr} to these restrictions, and averaging over all such intervals (using Minkowski's inequality and H\"older's inequality to estimate some error terms) it is not difficult to show that
$$
 \| \lim_{r \to 0} H_{k,r,R}( f_1, \dots, f_k ) \|_{L^p(\R)} \leq C'_{k,p,p_1,\dots,p_k} \|f_1\|_{L^{p_1}(\R)} \dots \|f_k\|_{L^{p_k}(\R)}
$$
for some constant $C'_{k,p,p_1,\dots,p_k}$ and test functions $f_1,\dots,f_k$, and then on subtracting this bound for two different choices of $R,r$ and using a limiting argument we obtain \eqref{hkp} (with a slightly worse constant).  We leave the details to the interested reader.

One can approach Conjecture \ref{tri} by introducing the operator norm $C_{k,p,p_1,\dots,p_k}(R/r)$ of $H_{k,r,R}$, defined as the best constant for which one has
$$ \| H_{k,r,R}( f_1, \dots, f_k ) \|_{L^p(\R)} \leq C_{k,p,p_1,\dots,p_k}(R/r) \|f_1\|_{L^{p_1}(\R)} \dots \|f_k\|_{L^{p_k}(\R)}$$
for all $f_1 \in L^{p_1}(\R),\dots,f_k \in L^{p_k}(\R)$.  Note from scaling that the operator norm of $H_{k,\lambda r, \lambda R}$ is the same as that of $H_{k,r,R}$ for any $\lambda > 0$, which is why we write the operator norm $C_{k,p,p_1,\dots,p_k}(R/r)$ as a function of the ratio $R/r$ rather than of $R,r$ separately. Conjecture \ref{tri} is then equivalent to the assertion that $C_{k,p,p_1,\dots,p_k}(R/r)$ remains bounded in the limit $R/r \to \infty$.  On the other hand, from \eqref{hkr}, Minkowski's integral inequality and H\"older's inequality we have the \emph{trivial bound}
\begin{align*}
C_{k,p,p_1,\dots,p_k}(R/r) &\leq \int_{r \leq |t| \leq R}\ \frac{dt}{|t|} \\
&= 2 \log \frac{R}{r}.
\end{align*}
Our main result is the following slight improvement of the trivial bound.

\begin{theorem}[Improvement over trivial bound]\label{main}  Let $k \geq 1$ and $1 < p_1,\dots,p_k,p < \infty$ be such that $\frac{1}{p} = \frac{1}{p_1} + \dots + \frac{1}{p_k}$, and let $\eps > 0$.   Then, if $R/r$ is sufficiently large depending on $\eps,k,p_1,\dots,p_k,p$, one has
$$ C_{k,p,p_1,\dots,p_k}(R/r)  \leq \eps \log \frac{R}{r}.$$
\end{theorem}

This falls well short of Conjecture \ref{tri}, but it does show that some cancellation is occurring in the $k$-linear Hilbert transform.  

A novel feature\footnote{See also \cite{christ} and \cite{kovac} for previous appearances of methods from arithmetic combinatorics in bounding multilinear operators related to $H_k$.} in our arguments is the introduction of tools from arithmetic combinatorics, particularly the theory of higher degree Gowers uniformity that was initially developed in \cite{gowers-4, gowers} to provide a new proof of Szemer\'edi's theorem \cite{szem} on arithmetic progressions.  Such tools are known to be useful for controlling expressions such as
$$ \sum_{x,t \in \Z} f_0(x) f_1(x+t) \dots f_k(x+kt) $$
 for various bounded, compactly supported functions $f_0,\dots,f_k: \Z \to \C$, so it is not so surprising in retrospect that they should also be able to say something non-trivial about integral expressions such as \eqref{hkr}.  Unfortunately, at the current state of development of the theory of higher degree uniformity, the quantitative bounds arising from these tools are quite poor for $k \geq 3$ (with no explicit bounds whatsoever in the current literature for $k \geq 5$), and so one would need a significant quantitative strengthening of the arithmetic combinatorics results, or the introduction of additional techniques, if one were to hope to make substantial progress towards Conjecture \ref{tri} beyond Theorem \ref{main}.

Roughly speaking, the strategy of proof of Theorem \ref{main} is as follows.  After some reductions reminiscent of those in \cite{lacey-1, lacey-2}, as well as a discretisation step in which one replaces the real line $\R$ with the integers $\Z$, one reduces matters to establishing a ``tree estimate'' in which one demonstrates non-trivial cancellation in expressions roughly of the form
\begin{equation}\label{tw}
 \sum_{x,t \in \Z} f_0(x) f_1(x+t) \dots f_k(x+kt) \psi( t / 2^n ) \varphi( 2^{-n} x - j ) 
\end{equation}
for some smooth compactly supported functions $\psi, \varphi$, some parameters $n,j$, and some bounded functions $f_0,\dots,f_k: \Z \to \R$.  Crucially, the function $\psi$ will be \emph{odd}, reflecting the odd nature of the Hilbert kernel $\frac{dt}{t}$ appearing in \eqref{hkdef}.  A standard ``generalised von Neumann theorem'' from arithmetic combinatorics tells us that expressions of the form \eqref{tw} are negligible if at least one of the functions $f_i$ is very small in a certain Gowers uniformity norm.  We then apply an \emph{arithmetic regularity lemma} from \cite{gt-reg} that asserts, roughly speaking, that any bounded function $f_i$ can be approximated (up to errors small in Gowers uniformity norm, plus an additional error small in an $L^2$ sense) with a special type of function, namely an \emph{irrational virtual nilsequence}.  This effectively allows one to replace all the functions $f_0,\dots,f_k$ in \eqref{tw} with such nilsequences.  The point of this reduction is that irrational nilsequences enjoy a \emph{counting lemma} (also from \cite{gt-reg}) that allows one to obtain good asymptotics for expressions such as \eqref{tw}.  At this point, the fact that $\psi$ is odd ensures that the main term in those asymptotics vanish, and the surviving error terms turn out to be small enough to eventually obtain the required conclusion in Theorem \ref{main}.

\subsection{Acknowledgments}

The author is supported by NSF grant DMS-1266164 and by a Simons Investigator Award.  He thanks Ciprian Demeter and Vjeko Kova{\v c} for helpful comments, and Pavel Zorin-Kranich for corrections.

\subsection{Notation}

We use the asymptotic notation $X \lesssim Y$, $Y \gtrsim X$, or $X = O(Y)$ to denote the assertion that $|X| \leq CY$ for some absolute constant $C$, which we call the \emph{implied constant}.  We will sometimes need to allow the implied constant to depend on additional parameters, in which case we indicate this by subscripts, e.g. $X \lesssim_\delta Y$, $Y \gtrsim_\delta Y$, or $X = O_\delta(Y)$ denotes the assertion that $|X| \leq C_\delta Y$ for some $C_\delta$ depending on $\delta$.  For brevity we will sometimes fix some basic parameters (e.g. $k$) and allow all implied constants to depend on such parameters (so that, for instance, $X = O_\delta(Y)$ is now short for $X = O_{\delta,k}(Y)$).  We also write $X \sim Y$ for $X \lesssim Y \lesssim X$.

We also use the asymptotic notation $X = o_{N \to \infty}(Y)$ to denote the assertion $|X| \leq c(N) Y$ where $c(N)$ is a quantity depending on a parameter $N$ that goes to zero as $N$ goes to infinity.  Again, if we need $c(N)$ to depend on external parameters, we will indicate this by subscripts; for instance, $X = o_{N \to \infty; k}(Y)$ denotes the assertion that $|X| \leq c_k(N) Y$ where $c_k(N)$ goes to zero as $N \to \infty$ for each fixed choice of $k$.

\section{Initial reductions}

We begin the proof of Theorem \ref{main}.   For technical reasons (having to do with the fact that the arithmetic regularity and counting lemmas in the literature are phrased in a discrete setting rather than a continuous one) we will need to transfer Theorem \ref{main} from the reals $\R$ to the integers $\Z$, giving up the scale invariance of the problem in the process.  Namely, we will derive Theorem \ref{main} from the following discrete version of that theorem.

\begin{theorem}[Discrete version of main theorem]\label{main-2}   Let $k \geq 1$ and $1 < p_1,\dots,p_k,p < \infty$ be such that $\frac{1}{p} = \frac{1}{p_1} + \dots + \frac{1}{p_k}$, and let $\eps > 0$.   Then, if $R \geq r \geq 1$ and $R/r$ is sufficiently large depending on $\eps,k,p_1,\dots,p_k,p$, one has
$$ 
\left\| \sum_{t \in \Z: r \leq |t| \leq R} \frac{f_1(x+t) \dots f_k(x+kt)}{t} \right\|_{\ell^p(\Z)} \leq \eps \log \frac{R}{r} \prod_{i=1}^k \|f_i\|_{\ell^{p_i}(\Z)} $$
for all $f_i \in \ell^{p_i}(\Z)$, $i=1,\dots,k$, where the $\ell^p$ norm on the left-hand side is in the $x$ variable.
\end{theorem}

Let us assume Theorem \ref{main-2} for the moment and see how it implies Theorem \ref{main}.  We will use a standard transference argument.  Let $k,p_1,\dots,p_k,p,\eps,R,r$ be as in Theorem \ref{main}, with $R/r$ assumed large enough depending on $p_1,\dots,p_k,p,\eps$.  Let $\lambda > 0$ be a large quantity (depending on $k,p_1,\dots,p,\eps,R,r$) to be chosen later.  For $\lambda$ large enough, we have $\lambda R \geq \lambda r \geq 1$, and so by Theorem \ref{main-2} we have
$$ 
\left\| \sum_{t \in \Z: \lambda r \leq |t| \leq \lambda R} \frac{f_1(x+t) \dots f_k(x+kt)}{t} \right\|_{\ell^p(\Z)} \leq \eps \log \frac{R}{r} \prod_{i=1}^k \|f_i\|_{\ell^{p_i}(\Z)} 
$$
for all $f_i \in \ell^{p_i}(\Z)$.  In particular, given $f_i \in L^{p_i}(\R)$ for $i=1,\dots,k$, we have
$$ 
\left\| \sum_{t \in \Z: \lambda r \leq |t| \leq \lambda R} \frac{f_1(x+t+\theta) \dots f_k(x+kt+\theta)}{t} \right\|_{\ell^p(\Z)} \leq \eps \log \frac{R}{r} \prod_{i=1}^k (\sum_{x \in \Z} |f_i(x+\theta)|^{p_i})^{1/p_i}
$$
for any $0 \leq \theta \leq 1$.  Averaging (in $L^p$) over all such $\theta$ and using H\"older's inequality and Fubini's theorem, we conclude that
$$ 
\left\| \sum_{t \in \Z: \lambda r \leq |t| \leq \lambda R} \frac{f_1(x+t) \dots f_k(x+kt)}{t} \right\|_{L^p(\R)} \leq \eps \log \frac{R}{r} \prod_{i=1}^k \|f_i\|_{L^{p_i}(\R)}.
$$
Rescaling by $\lambda$, we conclude that
$$ 
\left\| \frac{1}{\lambda} \sum_{t \in \frac{1}{\lambda} \Z: r \leq |t| \leq R} \frac{f_1(x+t) \dots f_k(x+kt)}{t} \right\|_{L^p(\R)} \leq \eps \log \frac{R}{r} \prod_{i=1}^k \|f_i\|_{L^{p_i}(\R)}.
$$
Sending $\lambda \to \infty$ and using Riemann integrability and Fatou's lemma, we conclude that
$$ 
\left\| \int_{r \leq |t| \leq R} f_1(x+t) \dots f_k(x+kt)\ \frac{dt}{t} \right\|_{L^p(\R)} \leq \eps \log \frac{R}{r} \prod_{i=1}^k \|f_i\|_{L^{p_i}(\R)}
$$
if the $f_i$ are continuous and compactly supported.  Applying a limiting argument, we obtain Theorem \ref{main}.

\begin{remark} In the converse direction, one can derive Theorem \ref{main-2} from Theorem \ref{main} by applying the latter to functions of the form $x \mapsto f_i( \lfloor x \rfloor )$; we leave the details to the interested reader.
\end{remark}

It remains to establish Theorem \ref{main-2}.  We will perform a number of preliminary reductions analogous to those in \cite{lacey-1}, \cite{lacey-2}, namely a reduction to a dyadic restricted weak-type estimate, and the construction of various ``trees'' of dyadic intervals, with the key nontrivial input being a tree estimate (see Proposition \ref{key} below) that improves over the trivial bound coming from the triangle inequality.

We turn to the details. By duality, it will suffice to show that
$$
\left|\sum_{t \in \Z: r \leq |t| \leq R} \sum_{x \in \Z} \frac{f_0(x) f_1(x+t) \dots f_k(x+(k-1)t)}{t}\right| \leq \eps \log \frac{R}{r} \prod_{i=0}^k \|f_i\|_{\ell^{p_i}(\R)}$$
whenever $k \geq 1$, $\eps>0$, $1 < p_0,\dots,p_k < \infty$ are such that $\frac{1}{p_0} + \dots + \frac{1}{p_k} = 1$, $f_i \in L^{p_i}(\R)$ for $i=0,\dots,k$, $R \geq r \geq 1$, and $R/r$ is sufficiently large depending on $\eps$.  By multilinear interpolation (and modifying $\eps,p_0,\dots,p_k$ as necessary), we may replace the strong Lebesgue norms $\ell^{p_i}(\Z)$ here by the Lorentz norms $\ell^{p_i,1}(\Z)$, and then by convexity we may reduce to the case where each of the $f_i$ are indicator functions, thus it will suffice to show that
\begin{equation}\label{lad}
\left|\sum_{t \in \Z: r \leq |t| \leq R} \sum_{x \in \Z} \frac{1_{E_0}(x) 1_{E_1}(x+t) \dots 1_{E_k}(x+(k-1)t)}{t}\right| \leq \eps \log \frac{R}{r} \prod_{i=0}^k |E_i|^{1/p_i}
\end{equation}
whenever $\eps>0$, $1 < p_0,\dots,p_k < \infty$ are such that $\frac{1}{p_0} + \dots + \frac{1}{p_k} = 1$, $E_i$ are subsets of $\Z$ with finite cardinality $|E_i|$, and $R \geq r \geq 1$ with $R/r$ is sufficiently large depending on $\eps$.  Here of course $1_E$ denotes the indicator function of $E$.  

Henceforth $k,p_0,\dots,p_k$ will be fixed, and all implied constants will be allowed to depend on these parameters.  From H\"older's inequality we have
$$
\sum_{x \in \Z}1_{E_0}(x) 1_{E_1}(x+t) \dots 1_{E_k}(x+(k-1)t)\ dx \leq \min_{0 \leq i \leq k} |E_i|$$
for any $t$, and thus
$$
\left|\sum_{t \in \Z: r \leq |t| \leq R} \sum_{x \in \Z} \frac{1_{E_0}(x) 1_{E_1}(x+t) \dots 1_{E_k}(x+(k-1)t)}{t}\right| \leq \log \frac{R}{r} \min_{0 \leq i \leq k} |E_i|.$$
Comparing this with \eqref{lad}, we see that we are done unless
\begin{equation}\label{ei-bound}
|E_i| \sim_\eps |E_0|
\end{equation}
for all $0 \leq i \leq k$.  Henceforth we will assume that \eqref{ei-bound} holds.  It will now suffice (after adjusting $\eps$ if necessary) to show that
\begin{equation}\label{lad-2}
\left|\sum_{t \in \Z: r \leq |t| \leq R} \sum_{x \in \Z} \frac{1_{E_0}(x) 1_{E_1}(x+t) \dots 1_{E_k}(x+(k-1)t)}{t}\right| \lesssim \eps \log \frac{R}{r} |E_0| 
\end{equation}
if $R/r$ is sufficiently large depending on $\eps$.

The next step is a decomposition into dyadic intervals.  Let $\psi: \R \to \R$ be a fixed smooth odd function supported on $[-2,-1/2] \cup [1/2,2]$ with the property that
$$ \sum_{n \in \Z} 2^{-n} \psi(2^{-n} t) = \frac{1}{t}$$
for all $t \neq 0$; such a function can be constructed by taking $\psi(t) := \frac{\phi(t) - \phi(t/2)}{t}$ for some smooth even $\phi: \R \to \R$ supported on $[-1,1]$ and equaling $1$ on $[-1/2,1/2]$.  Henceforth we allow implied constants to depend on $\psi$.  Then the function $1_{r \leq |t| \leq R} \frac{1}{t}$ differs from $\sum_{n: r \leq 2^{n} \leq R} 2^{-n} \psi(t/2^n)$ only when $t \sim R$ or $t \sim r$, where both functions are $O(1/R)$ and $O(1/r)$ respectively.  From this and the triangle inequality one sees that \eqref{lad-2} is equivalent to
$$
\left|\sum_{n: r \leq 2^{n} \leq R} 2^{-n} \sum_{x,t \in \Z} 1_{E_0}(x) 1_{E_1}(x+t) \dots 1_{E_k}(x+kt) \psi(t/2^n)\right| \lesssim \eps \log \frac{R}{r} |E_0|
$$
since the left-hand side here differs from that of \eqref{lad-2} by $O(1)$, which is acceptable if $R/r$ is large enough depending on $\eps$.  By the triangle inequality, it thus suffices to show that
$$
\sum_{n: r \leq 2^{n} \leq R} 2^{-n} \left|\sum_{x,t \in \Z} 1_{E_0}(x) 1_{E_1}(x+t) \dots 1_{E_k}(x+kt) \psi(t/2^n) \ dx dt\right| \lesssim \eps \log \frac{R}{r} |E_0|. 
$$
We now introduce a further smooth function $\varphi: \R \to \R$ supported on $[-1,1]$ such that
$$ \sum_{j \in \Z} \varphi(x - j ) = 1$$
for all $x \in \R$; indeed one can take $\varphi(x) := \eta(x) - \eta(x+1)$ for some smooth $\eta: \R \to \R$ equal to $1$ for negative $x$ and $0$ for $x>1$.  We allow implied constants to depend on $\varphi$.  For each\footnote{As we are working on the integers, we do not consider dyadic intervals of length less than $1$.} (discrete) dyadic interval $I = \{ x \in \Z: j 2^n < x \leq (j+1) 2^n \}$ with $n \geq 0$, we define the quantity
\begin{equation}\label{aie}
 a_I := 2^{-2n} \left|\sum_{x,t \in \Z} 1_{E_0}(x) 1_{E_1}(x+t) \dots 1_{E_k}(x+(k-1)t) \psi(t/2^n) \varphi( 2^{-n} x - j )\ dx dt\right|
\end{equation}
so by the triangle inequality it suffices to show that
\begin{equation}\label{targ}
 \sum_{I: r \leq |I| \leq R} a_I |I| \lesssim \eps \log \frac{R}{r} |E_0|
\end{equation}
where the sum is over dyadic intervals $I$ of length between $r$ and $R$.

From the triangle inequality and \eqref{aie} we have the bound
\begin{equation}\label{aio}
a_I \lesssim 1
\end{equation}
for all $I$.  We also have the following estimate:

\begin{lemma}\label{law}  We have
$$ \sum_{I: r \leq |I| \leq R} a_I^{p/2} |I| \lesssim_{\eps,p} \log \frac{R}{r} |E_0|$$
for any $1 < p < \infty$.
\end{lemma}

\begin{proof}  We bound
\begin{align*}
a_I &\lesssim 2^{-2n} \int_{|t| \sim 2^n} \int_{x = (j+O(1))2^n} 1_{E_0}(x) 1_{E_1}(x+t)\ dx dt \\
&\lesssim 2^{-2n} \left(\int_{x = (j+O(1))2^n} 1_{E_0}(x)\ dx\right) \left(\int_{y = (j+O(1))2^n} 1_{E_1}(y)\ dy\right) \\
&\lesssim \inf_{x \in I} \M 1_{E_0}(x) \M 1_{E_1}(x) 
\end{align*}
where $\M f(x) := \sup_{r \geq 1} \frac{1}{2r+1} \sum_{y \in \Z: |y-x| \leq r} |f(y)|$ is the Hardy-Littlewood maximal operator on the integers $\Z$.  Thus
$$ a_I^{p/2} |I| \lesssim \sum_{x \in I} \M 1_{E_0}(x)^{p/2} \M 1_{E_1}(x)^{p/2} $$
and so (since each $x \in \Z$ is contained in $O( \log \frac{R}{r} )$ dyadic intervals $I$ with $r \leq |I| \leq R$)
$$ \sum_{I: r \leq |I| \leq R} a_I^{p/2} |I| \lesssim \left(\log \frac{R}{r}\right) \sum_{x \in \Z} \M 1_{E_0}(x)^{p/2} \M 1_{E_1}(x)^{p/2} $$
and the claim follows from the Cauchy-Schwarz inequality, the Hardy-Littlewood maximal inequality, and \eqref{ei-bound}.
\end{proof}

From the above lemma with $p=3/2$ (say) we see in particular that
$$ \sum_{I: r \leq |I| \leq R; a_I \leq \delta} a_I |I| \lesssim_\eps \delta^{1/4} \log \frac{R}{r} |E_0|$$
so to prove \eqref{targ} it suffices by \eqref{aio} to show that
\begin{equation}\label{epsdel}
\sum_{I: r \leq |I| \leq R; a_I > \delta} |I| \lesssim \eps \log\frac{R}{r} |E_0|
\end{equation}
for any $\delta > 0$, whenever $R/r$ is sufficiently large depending on $\eps,\delta$.

In the next section, we will establish the following result.

\begin{proposition}[Cancellation in a tree]\label{key}  Let $I_0$ be a dyadic interval, and let $\delta>0$.  Then there exists a quantity $1 \leq A \ll_{\delta} 1$ (which can depend on $I_0$) such that
$$ \sum_{I: I \subset I_0; |I_0|/A \leq |I| \leq |I_0|} a_I |I| \lesssim \delta |I_0| \log A.$$
\end{proposition}

The point here, of course, is the gain of $\delta$ on the right-hand side, since otherwise the claim is immediate from the trivial bound \eqref{aio}.

Let us assume this proposition for the moment and see how to conclude \eqref{epsdel}.  Call a dyadic interval $I$ \emph{bad} if $r \leq |I| \leq R$ and $a_I > \delta$.  From the proof of Lemma \ref{law} (bounding $M 1_{E_1}$ crudely by $1$) we see that
$$ \M 1_{E_0}(x) \gg \delta $$
whenever $x$ lies in a bad interval.  In particular, from the Hardy-Littlewood maximal inequality we see that there are only finitely many bad intervals.

Let ${\mathcal I}$ denote the collection of bad dyadic intervals.  Using a greedy algorithm (starting with the largest bad intervals and only moving on to the smaller bad intervals once all the largest onces have been covered), as well as Proposition \ref{key} (with $\delta$ replaced by $\eps \delta$), we may cover ${\mathcal I}$ by a family ${\mathcal T}$ of disjoint ``trees'' $T$, each of the form $T = \{ I: I \subset I_T; |I_T|/A_T \leq |I| \leq |I_T| \}$ for some ``tree top'' $I_T \in {\mathcal I}$ and some quantity $1 \leq A_T \ll_{\eps,\delta} 1$, with the property that
$$ \sum_{I \in T} a_I |I| \lesssim \eps \delta |I_T| \log A_T.$$
Note that we only require that the top $I_T$ of the tree $T$ lie in ${\mathcal I}$; the other elements of $T$ may lie outside ${\mathcal I}$.

Summing over all trees $T \in {\mathcal T}$, we conclude that
$$ \sum_{I: r \leq |I| \leq R; a_I > \delta} |I| \lesssim \sum_{T \in {\mathcal T}} \eps \delta |I_T| \log A_T.$$
On the other hand, we have
$$ |I_T| \log A_T \ll \int_\R \sum_{I \in T} 1_I(x)\ dx $$
for each tree $T \in {\mathcal T}$, and thus (by the disjointness of the trees $T$)
$$ \sum_{I: r \leq |I| \leq R; a_I > \delta} |I| \lesssim \eps \delta \sum_{x \in \Z} \sum_{I \in \bigcup_{T \in {\mathcal T}} T} 1_I(x)\ dx.$$
For each $x$ in the support of $\sum_{I \in \bigcup_{T \in {\mathcal T}} T} 1_I(x)$, we have $x \in I_T$ for some tree top $I_T$, and thus $\M 1_{E_0}(x) \gg \delta$.  By the Hardy-Littlewood maximal inequality, we thus see that $x$ is contained in a set of measure $O(|E_0|/\delta)$.  Finally, by construction, every interval $I$ in a tree $T \in {\mathcal T}$ has size at most $R$ and at least $r / A_T \gg_{\eps,\delta} r$, and so each $x$ is contained in at most $O( \log \frac{R}{r} )$ intervals if $R/r$ is sufficiently large depending on $\eps,\delta$.  Putting all this together we obtain \eqref{epsdel} as required.

It remains to establish Proposition \ref{key}.  This will be accomplished in the next section.  

\section{Applying the arithmetic regularity and counting lemmas}

By translation we may assume that
$$ I_0 = \{1,\dots,N\} =: [N]$$ 
for some natural number $N$ which is a power of $2$.  Our task is to find $1 \leq A \ll_\delta 1$ such that
\begin{equation}\label{nan}
\begin{split}
&\sum_{n: N/A \leq 2^n \leq N} \sum_{j=0}^{N/2^n - 1}
2^{-n} \left|\sum_{x,t \in \Z} 1_{E_0}(x) 1_{E_1}(x+t) \dots 1_{E_k}(x+(k-1)t) \psi(t/2^n) \varphi( 2^{-n} x - j )\right| \\
&\quad \lesssim \delta N \log A.
\end{split}
\end{equation}
We will in fact produce an $A$ with $A \geq C_\delta$, where $C_\delta$ is a sufficiently large quantity depending on $\delta$.  We may assume that $N$ is sufficiently large depending on $\delta$, since otherwise we can use the trivial bound of $O(N \log N)$ on the left-hand side (coming from the fact that there are only $O(\log N)$ choices for $n$) to conclude, after choosing $A$ large enough.

Let $E'_i := E_i \cap [N]$.  For each $n$ with $N/A \leq 2^n \leq N$, we may replace the $E_i$ by $E'_i$ in \eqref{nan} for all but $O(1)$ choices of $j$ (coming from those $j$ near $0$ or $N/2^n$).  The total error in replacing $E_i$ by $E'_i$ is thus
$$ \sum_{n: N/A \leq 2^n \leq N} O( 2^{-n} 2^{2n} ) = O(N) $$
which is acceptable since $A \geq C_\delta$ for some large $C_\delta$.  It thus suffices to show that
\begin{equation}\label{nan-2}
\begin{split}
&\sum_{n: N/A \leq 2^n \leq N} \sum_{j=0}^{N/2^n - 1}
2^{-n} \left|\sum_{x,t \in \Z} 1_{E'_0}(x) 1_{E'_1}(x+t) \dots 1_{E'_k}(x+kt) \psi(t/2^n) \varphi( 2^{-n} x - j )\right|\\
&\quad \lesssim \delta N \log A
\end{split}
\end{equation}
for some $C_\delta \leq A \ll_\delta 1$.

To control this expression we recall the \emph{Gowers uniformity norms} from \cite{gowers-4,gowers}.  If $f: \Z/N\Z \to \C$ is a function on a cyclic group $\Z/N\Z$ and $d \geq 1$, we define the Gowers uniformity norm $\|f\|_{U^d(\Z/N\Z)}$ by the formula
$$ \|f\|_{U^d(\Z/N\Z)} = \left( \frac{1}{N^{d+1}} \sum_{h_1,\dots,h_d,x \in \Z/N\Z} \Delta_{h_1} \dots \Delta_{h_d} f(x) \right)^{1/2^d} $$
where $\Delta_h f(x) := f(x+h) \overline{f(x)}$.  Given instead a function $f: [N] \to \C$ on the interval $[N]$, we define the Gowers norm $\|f\|_{U^d([N])}$ by the formula
$$ \|f\|_{U^d([N])} =\|f\|_{U^d(\Z/N'\Z)} / \|1_{[N]}\|_{U^d(\Z/N'\Z)} $$
for any $N' \geq 2^d N$, where we embed $[N]$ into $\Z/N'\Z$ and extend $f$ by zero outside of $[N]$; it is easy to see that this definition does not depend on the choice of $N$.  We have the \emph{generalised von Neumann theorem}
\begin{equation}\label{jove}
\frac{1}{N^2} \left| \sum_{x,t \in \Z} f_0(x) f_1(x+t) \dots f_k(x+kt) \right| \lesssim \inf_{0 \leq i \leq k} \|f_i\|_{U^{k}([N])} 
\end{equation}
for any $f_0,\dots,f_k: [N] \to \C$ bounded in magnitude by $1$ (extending by zero outside of $[N]$); see e.g. \cite[Lemma 11.4]{tao-vu} (after applying the embedding of $[N]$ into some suitable $\Z/N'\Z$).  We have the following variant:

\begin{lemma}\label{unf-con}  Let $d := \max(k,2)$.  Then for any $n$ with $2^n \leq N$, any $j \in \Z$ and any functions $f_0,\dots,f_k: [N] \to \C$ bounded in magnitude by $1$ (and extended by zero outside of $[N]$, we have
$$
\frac{1}{N^2} \left|\sum_{x,t \in \Z} f_0(x) f_1(x+t) \dots f_k(x+kt) \psi(t/2^n) \varphi( 2^{-n} x - j )\right| \lesssim 
\inf_{0 \leq i \leq k} \|f_i\|_{U^{d}([N])}.$$
In particular, by the triangle inequality we have
\begin{equation}\label{sod}
\begin{split}
&\sum_{n: N/A \leq 2^n \leq N} \sum_{j=0}^{N/2^n - 1}
2^{-n} \left|\sum_{x,t \in \Z} f_0(x) f_1(x+t) \dots f_k(x+kt) \psi(t/2^n) \varphi( 2^{-n} x - j )\right|\\
&\quad \lesssim_A N \inf_{0 \leq i \leq k} \|f_i\|_{U^{d}([N])}.
\end{split}
\end{equation}
\end{lemma}

\begin{proof} By adding a dummy function $f_{k+1}$ if necessary we may assume that $k \geq 2$, so that $d=k$.  By Fourier inversion we may write
$$ \psi(t) = \int_\R e^{2\pi itu} \hat \psi(u)\ du$$
and
$$ \varphi(x) = \int_\R e^{2\pi i x v} \hat \varphi(v)\ dv$$
for some rapidly decreasing (and in particular, absolutely integrable) functions $\hat \psi, \hat \varphi: \R \to \C$.  By the triangle inequality, it thus suffices to show that
$$
\frac{1}{N^2} \left|\sum_{x,t \in \Z} f_0(x) f_1(x+t) \dots f_k(x+kt) e^{2\pi i tu} e^{2\pi i xv}\right| \lesssim 
\inf_{0 \leq i \leq k} \|f_i\|_{U^{k}([N])}$$
uniformly for all $u,v \in \R$.
Writing
$$ f_0(x) f_1(x+t) e^{2\pi i tu} e^{2\pi i xv} = f_0(x) e^{2\pi i x(v-u)} \times f_1(x+t) e^{2\pi i (x+t) u} $$
and applying \eqref{jove}, we see that
$$
\frac{1}{N^2} \left|\sum_{x,t \in \Z} f_0(x) f_1(x+t) \dots f_k(x+kt) e^{2\pi i tu} e^{2\pi i xv}\right| \lesssim 
\inf_{0 \leq i \leq k} \|\tilde f_i \|_{U^{k}([N])}$$
where $\tilde f_i$ is $f_i$ modulated by a Fourier character $x \mapsto e^{2\pi i x v}$ for some $v \in \R$.  But if $k \geq 2$, one easily verifies that $\tilde f_i$ has the same $U^k$ norm as $f_i$ (because $\Delta_{h_1} \Delta_{h_2} \tilde f_i = \Delta_{h_1} \Delta_{h_2} f_i$ for any $h_1,h_2$), and the claim follows.
\end{proof}

We also need a similar statement in which the $U^d$ norm is replaced by the $L^2$ norm:

\begin{lemma}\label{l2-con}  For any $f_0,\dots,f_k:[N] \to \C$ bounded in magnitude by $1$, we have
\begin{align*}
&\sum_{n: N/A \leq 2^n \leq N} \sum_{j=0}^{N/2^n - 1}
2^{-n} \left|\sum_{x,t \in \Z} f_0(x) f_1(x+t) \dots f_k(x+kt) \psi(t/2^n) \varphi( 2^{-n} x - j )\right|\\
&\quad \lesssim N^{1/2} (\log A) \inf_{0 \leq i \leq k} \|f_i\|_{\ell^2([N])}.
\end{align*}
\end{lemma}

\begin{proof}  Let $0 \leq i \leq k$.  Observe that for each integer $y$ and each $n$ with $N/A \leq 2^n \leq N$, there are at most $O(1)$ choices of $j$ for which the sum $\sum_{x,t \in \Z} f_0(x) f_1(x+t) \dots f_k(x+kt) \psi(t/2^n) \varphi( 2^{-n} x - j )$ contains a non-zero term involving $f_i(y)$, and when this is the case there are $O(2^n)$ such terms.  From this and the triangle inequality we see that
\begin{align*}
&\sum_{n: N/A \leq 2^n \leq N} \sum_{j=0}^{N/2^n - 1}
2^{-n} \left|\sum_{x,t \in \Z} f_0(x) f_1(x+t) \dots f_k(x+kt) \psi(t/2^n) \varphi( 2^{-n} x - j )\right|\\
&\quad \lesssim (\log A) \sum_{y \in \Z} |f_i(y)|
\end{align*}
and the claim now follows from the Cauchy-Schwarz inequality.
\end{proof}

The above lemmas show, roughly speaking, that we may freely modify each of the $1_{E'_i}$ by errors that are either small in (normalised) $\ell^2([N])$ norm, or \emph{extremely} small in $U^d([N])$ norm.  To exploit this phenomenon, we will need the \emph{arithmetic regularity lemma} from \cite{gt-reg} that asserts that all bounded functions on $[N]$, up to errors of the above form, can be expressed as a very well distributed nilsequence.  To make this statement precise, we need to recall a large number of definitions from \cite{gt-reg} (although for the purposes of this paper, many of the concepts defined here can be taken as ``black boxes'').

\begin{definition}[Filtered nilmanifold] Let $s \geq 1$ be an integer.  A \emph{filtered nilmanifold} $G/\Gamma = (G/\Gamma, G_\bullet)$ of degree $\leq s$ consists of the following data:
\begin{enumerate}
\item A connected, simply-connected nilpotent Lie group $G$;
\item A discrete, cocompact subgroup $\Gamma$ of $G$ (thus the quotient space $G/\Gamma$ is a compact manifold, known as a \emph{nilmanifold});
\item A \emph{filtration} $G_\bullet = (G_{(i)})_{i=0}^\infty$ of closed connected subgroups
$$ G = G_{(0)} = G_{(1)} \geq G_{(2)} \geq \ldots$$
of $G$, which are \emph{rational} in the sense that the subgroups $\Gamma_{(i)} := \Gamma \cap G_{(i)}$ are cocompact in $G_{(i)}$, such that $[G_{(i)},G_{(j)}] \subseteq G_{(i+j)}$ for all $i,j \geq 0$, and such that $G_{(i)}=\{\id\}$ whenever $i>s$;
\item A \emph{Mal'cev basis} $\mathcal{X} = (X_1,\ldots,X_{\dim(G)})$ adapted to $G_{\bullet}$, that is to say a basis $X_1,\ldots,X_{\dim(G)}$ of the Lie algebra of $G$ that exponentiates to elements of $\Gamma$, such that $X_j,\ldots,X_{\dim(G)}$ span a Lie algebra ideal for all $j \leq i \leq \dim(G)$, and $X_{\dim(G)-\dim(G_{(i)})+1},\ldots,X_{\dim(G)}$ spans the Lie algebra of $G_{(i)}$ for all $1 \leq i \leq s$. 
\end{enumerate}
\end{definition}

One may use a Mal'cev basis to define a metric $d_{G/\Gamma}$ on the nilmanifold $G/\Gamma$, as per \cite[Definition 2.2]{green-tao-nilratner}. 

\begin{definition}[Complexity] Let $M \geq 1$.  We say that a filtered nilmanifold $G/\Gamma = (G/\Gamma,G_\bullet)$ has \emph{complexity $\leq M$} if the dimension of $G$, the degree of $G_\bullet$, and the rationality of the Mal'cev basis $\mathcal{X}$ (cf. \cite[Definition 2.4]{green-tao-nilratner}) are bounded by $M$.
\end{definition}

\begin{definition}[Polynomial sequence]  Let $(G/\Gamma,G_\bullet)$ be a filtered nilmanifold, with filtration $G_\bullet = (G_{(i)})_{i=0}^\infty$.  A \emph{polynomial sequence} adapted to this filtered nilmanifold is a sequence $g: \Z \to G$ with the property that 
$$ \partial_{h_1} \ldots \partial_{h_i} g(n) \in G_{(i)}$$
for all $i \geq 0$ and $h_1,\ldots,h_i,n \in \Z$, where $\partial_h g(n) := g(n+h) g(n)^{-1}$ is the derivative of $g$ with respect to the shift $h$. 
\end{definition}

\begin{definition}[Orbits]  Let $s \geq 1$ be an integer, and let $M, A > 0$ be parameters.  A  \emph{polynomial orbit} of degree $\leq s$ and complexity $\leq M$ is any function $n \mapsto g(n) \Gamma$ from $\Z \to G/\Gamma$, where $(G/\Gamma,G_\bullet)$ is a filtered nilmanifold of complexity $\leq M$, and $g: \Z \to G$ is a polynomial sequence.  
\end{definition}

\begin{definition}[Nilsequences]
A \emph{\textup{(}polynomial\textup{)} nilsequence} of degree $\leq s$ and complexity $\leq M$ is any function $f: \Z \to \C$ of the form $f(n) = F(g(n)\Gamma)$, where $n \mapsto g(n)\Gamma$ is a polynomial orbit of degree $\leq s$ and complexity $\leq M$, and $F: G/\Gamma \to \C$ is a function of Lipschitz norm\footnote{The (inhomogeneous) Lipschitz norm $\|F\|_{\operatorname{Lip}}$ of a function $F: X \to \C$ on a metric space $X = (X,d)$ is defined as
$$\|F\|_{\operatorname{Lip}} := \sup_{x \in X} |F(x)| + \sup_{x,y \in X: x \neq y} \frac{|F(x)-F(y)|}{|x-y|}.$$
} at most $M$.
\end{definition}

\begin{definition}[Virtual nilsequences]\label{virt-def}
Let $N \geq 1$.  A \emph{virtual nilsequence} of degree $\leq s$ and complexity $\leq M$ at scale $N$ is any function $f: [N] \to \C$ of the form $f(n) = F(g(n)\Gamma, n \md{q}, n/N)$, where $1 \leq q \leq M$ is an integer, $n \mapsto g(n)\Gamma$ is a polynomial orbit of degree $\leq s$ and complexity $\leq M$, and $F: G/\Gamma \times \Z/q\Z \times \R \to \C$ is a function of Lipschitz norm at most $M$.  (Here we place a metric on $G/\Gamma \times \Z/q\Z \times \R$ in some arbitrary fashion, e.g. by embedding $\Z/q\Z$ in $\R/\Z$ and taking the direct sum of the metrics on the three factors.)  We define a \emph{vector valued virtual nilsequence} $f: [N] \to \C^d$ similarly, except that $F$ now takes values in $\C^d$ instead of $\C$.
\end{definition}

We now have almost all the definitions needed to state the arithmetic regularity lemma:

\begin{theorem}[Arithmetic regularity lemma]\label{strong-reg}
Let $f: [N] \to [0,1]^d$ be a function for some $d \geq 0$, let $s \geq 1$ be an integer, let $\eps > 0$, and let $\Grow: \R^+ \to \R^+$ be a monotone increasing function with $\Grow(M) \geq M$ for all $M$.  Then there exists a quantity $M = O_{s,\eps,\Grow,d}(1)$ and a decomposition 
$$ f = f_{\nil} + f_{\sml} + f_{\unf}$$
of $f$ into functions $f_{\nil}, f_{\sml}, f_{\unf}: [N] \to [-1,1]^d$ of the following form:
\begin{enumerate}
\item \textup{(}$f_{\nil,i}$ structured\textup{)} $f_{\nil}$ is a $(\Grow(M),N)$-irrational vector-valued virtual nilsequence of degree $\leq s$, complexity $\leq M$, and scale $N$, where the notion of $(A,N)$-irrationality is defined in \cite[Definition A.6]{gt-reg};
\item \textup{(}$f_\sml$ small\textup{)} Each of the $d$ components of $f_{\sml}$ has an $\ell^2([N])$ norm of at most $\eps N^{1/2}$;
\item \textup{(}$f_\unf$ very uniform\textup{)} Each of the $d$ components of $f_{\unf}$ has a $U^{s+1}([N])$ norm of at most $1/\Grow(M)$;
\item \textup{(}Nonnegativity\textup{)} $f_{\nil}$ and $f_{\nil} + f_{\sml}$ take values in $[0,1]^d$.
\end{enumerate}
\end{theorem}

\begin{proof}  See \cite[Theorem 1.2]{gt-reg}.  Strictly speaking, this theorem is only stated in the scalar case $d=1$, but the same argument extends without difficulty to the vector-valued case $d \geq 1$ (note we allow our bound on $M$ to depend on $d$).  We remark that the bounds on $M$ in the above theorem are extremely poor (at least tower-exponential or worse, in practice) for a variety of reasons, including the lack of good (or indeed any) quantitative bounds for the inverse theorem for higher order Gowers uniformity norms.
\end{proof}

We apply this lemma with $\eps$ replaced by $\delta^2$, $s+1$ replaced by $d$, $d$ replaced by $k+1$, and $\Grow$ a rapidly increasing function to be chosen later, to obtain decompositions
$$ 1_{E'_i} = f_{\nil,i} + f_{\sml,i} + f_{\unf,i} $$
with $f_{\nil,i}, f_{\sml,i}, f_{\unf,i}: [N] \to [-1,1]$ being the components of functions $f_{\nil}, f_{\sml}, f_{\unf}: [N] \to [-1,1]^{k+1}$ obeying the conclusions of Theorem \ref{strong-reg}.  
By the triangle inequality, the left-hand side of \eqref{nan-2} can be written as the sum of $3^{k+1} = O(1)$ terms, in which each of the $1_{E'_i}$ has been replaced by one of $f_{\nil,i}, f_{\sml,i}, f_{\unf,i}$.  By Lemma \ref{unf-con}, the contribution of any term involving one of the $f_{\unf,i}$ is at most $O_A( N / \Grow(M) )$, while from Lemma \ref{l2-con} the contribution of any term involving one of the $f_{\sml,i}$ is $O(\delta N \log A)$.  Thus, one may bound the left-hand side of \eqref{nan-2} by
$$
\left( \sum_{n: N/A \leq 2^n \leq N} X_n \right) + O_A( N / \Grow(M) ) + O(\delta N \log A)$$
where $X_n$ is the quantity 
$$ X_n := \sum_{j=0}^{N/2^n - 1} 2^{-n} \left|\sum_{x,t \in \Z} f_{\nil,0}(x) f_{\nil,1}(x+t) \dots f_{\nil,k}(x+kt) \psi(t/2^n) \varphi( 2^{-n} x - j )\right|.$$
We now turn to the estimation of $X_n$.  Bounding each $f_{\nil,i}$ by $O(1)$, we have the trivial bound
\begin{equation}\label{xan}
 X_n \lesssim N 
\end{equation}
which we will use for values of $2^n$ that are close to $N$.  For the remaining values of $n$, we argue as follows.  By Definition \ref{virt-def}, we have
$$f_{\nil,i}(x+it) = F_i\left(g(x+it)\Gamma, x+it \md{q}, \frac{x+it}{N}\right)$$
whenever $0\leq i \leq k$ and $x+it \in [N]$, for some positive integer $q = O_M(1)$, some $(\Grow(M),N)$-irrational polynomial orbit $n \mapsto g(n)\Gamma$ of degree $\leq d-1$ and complexity $O_M(1)$ into a filtered nilmanifold $G/\Gamma$ of degree $\leq d-1$ and complexity $O_M(1)$, and a function $F_i: G/\Gamma \times \Z/q\Z \times \R \to \C$ of Lipschitz norm at most $O_M(1)$. 

For all but $O(1)$ values of $0 \leq j \leq N/2^n-1$, one has $x,x+t,\dots,x+kt \in [N]$ in the support of $\psi(t/2^n) \varphi( 2^{-n} x - j )$.  The exceptional values of $j$ contribute $O(2^n)$ to $X_n$; thus
$$ X_n \ll \sum_{j=0}^{N/2^n - 1} 2^{-n} \left|\sum_{x,t \in \Z} \prod_{i=0}^k F_i\left(g(x+it)\Gamma, x+it \md{q}, \frac{x+it}{N}\right) \psi(t/2^n) \varphi( 2^{-n} x - j )\right| + 2^n.$$
By the triangle inequality, we thus have
$$ X_n \ll 2^{-2n} N \left|\sum_{x,t \in \Z} \prod_{i=0}^k F_i(g(x+it)\Gamma, x+it \md{q}, \frac{x+it}{N}) \psi(t/2^n) \varphi( 2^{-n} x - j )\right| + 2^n$$
for some $0 \leq j < N/2^n$.
Splitting $x,t$ into residue classes modulo $q = O_M(1)$ and using the triangle inequality, we thus have
$$ X_n \ll_M 2^{-2n} N  \left|\sum_{x,t \in \Z; x = a \md{q}, t = b \md{q}} \prod_{i=0}^k F_i(g(x+it)\Gamma, a+ib \md{q}, \frac{x+it}{N}) \psi(t/2^n) \varphi( 2^{-n} x - j )\right| + 2^n$$
for some $0 \leq a,b < q$.  Next, on the support of $\psi(t/2^n) \varphi( 2^{-n} x - j )$ we have $\frac{x+it}{N} = \frac{j 2^n}{N} + O( \frac{2^n}{N} )$, hence by the Lipschitz property
$$ F_i(g(x+it)\Gamma, a+ib \md{q}, \frac{x+it}{N}) = F'_{i}(g(x+it)\Gamma) + O_M\left( \frac{2^n}{N} \right) $$
where $F'_{i} = F'_{i,j,a,b,n,N}: G/\Gamma \to \R$ is the function
$$ F'_{i}(x) := F_i\left(x, a+ib \md{q}, \frac{j 2^n}{N} \right).$$
Note that $F'_{i}$ has a Lipschitz norm of $O_M(1)$.
We conclude that
$$ X_n \ll_M 2^{-2n} N \left|\sum_{x,t \in \Z; x = a \md{q}, t = b \md{q}} \prod_{i=0}^k F_{i,j}(g(x+it)\Gamma) \psi(t/2^n) \varphi( 2^{-n} x - j )\right| + 2^n.$$
Making the substitution $x = qx' + a$, $t = qt' + b$, this becomes
$$ X_n \ll_M 2^{-2n} N \left|\sum_{x',t' \in \Z} \prod_{i=0}^k F_{i,j}(g(qx'+iqt' + a+ib)\Gamma) \psi(\frac{qt'+a}{2^n}) \varphi( 2^{-n} q x' + 2^{-n} a - j )\right| + 2^n.$$
At this point, we use the counting lemma from \cite[Theorem 1.11]{gt-reg}.  This gives the asymptotic
$$
\sum_{x' \in I, t' \in J} \prod_{i=0}^k F_{i,j}(g(qx'+iqt' + a+ib)\Gamma) = 
|I| |J| \alpha + o_{\Grow(M) \to \infty;M}(N^2) + o_{N \to \infty; M}(N^2) $$
for any intervals $I, J \subset [-N,N]$, where $\alpha$ is a quantity independent of $I,J$ (it is given by an explicit integral of a certain Lipschitz function on a certain filtered nilmanifold, but its precise value is immaterial for the current argument).  Since the left-hand side of this asymptotic is $O_M( |I| |J| )$, we may assume without loss of generality that $\alpha = O_M(1)$.  A routine Riemann sum argument using the bound $2^n \geq N/A$ and the smooth nature of $\psi,\varphi$ (decomposing the $x',t'$ variables into intervals of length $[N/A^{100}]$) then shows that
\begin{align*}
&\sum_{x',t' \in \Z} \prod_{i=0}^k F_{i,j}(g(qx'+iqt' + a+ib)\Gamma) \psi(\frac{qt'+a}{2^n}) \varphi( 2^{-n} q x' + 2^{-n} a - j ) \\
&\quad = \alpha \int_{\R^2} \psi(\frac{qt'+a}{2^n}) \varphi( 2^{-n} q x' + 2^{-n} a - j )\ dx' dt' + O( A^{-90} N^2 ) \\
&\quad\quad +  o_{\Grow(M) \to \infty;M,A}(N^2) + o_{N \to \infty; M,A}(N^2).
\end{align*}
Since $\psi$ is odd, the integral here vanishes.  We conclude (since $2^{-2n} N \ll A^2 N^{-1}$) that
$$ X_n \ll_M o_{\Grow(M) \to \infty;M,A}(N) + o_{N \to \infty; M,A}(N) + A^{-80} N + 2^n.$$
Using this bound for $2^n \leq A^{-\delta} N$, and the trivial bound \eqref{xan} for $A^{-\delta} N < n \leq N$, we have
$$
\sum_{n: N/A \leq 2^n \leq N} X_n \ll o_{\Grow(M) \to \infty;M,A}(N) + o_{N \to \infty; M,A}(N) + O_M( A^{-\delta} N ) + \delta N \log A
$$
and so we may bound the left-hand side of \eqref{nan-2} by
$$
o_{\Grow(M) \to \infty;M,A}(N) + o_{N \to \infty; M,A}(N) + O_M( A^{-\delta} N ) + O( \delta N \log A ).
$$
If we choose $A$ sufficiently large depending on $\delta, M$, and then $\Grow$ sufficiently rapidly growing, and then $N$ sufficiently large depending on $\delta, \Grow$ (recalling that $M = O_{\delta, \Grow}(1)$), we can make this expression $O( \delta N \log A )$, giving \eqref{nan-2} as required.

\begin{remark}  The above arguments also provide some non-trivial cancellation for (truncations of) other variants of the multilinear Hilbert transform.  For instance, one replace $H_k$ by the maximal truncated $k$-linear Hilbert transform
$$ \sup_{r < R} |H_{k,r,R} f(x)|;$$
one could also consider the multilinear Hilbert transform
\begin{equation}\label{cck}
\operatorname{p.v.} \int_\R f_1(x+c_1 t) \dots f_k(x+c_k t)\ \frac{dt}{t}
\end{equation}
with rational coefficients $c_1,\dots,c_k$; we leave the modification of the above arguments to these operators to the interested reader.  Curiously, there appears to be some difficulty extending the arguments to the final variant \eqref{cck} if the $c_1,\dots,c_k$ are not rationally commensurate, as one cannot easily discretise in this case to deploy additive combinatorics tools.  A somewhat similar phenomenon appeared previously in \cite{christ}.  It may be possible to get around this difficulty by developing a continuous version of the arithmetic regularity and counting lemmas.  It is also possible, in principle at least, to obtain a corresponding result for the maximal multilinear Hilbert transform
or by a polynomial Carleson type operator
$$
\sup_P \operatorname{p.v.} \int_{\R} f_1(x+t) \dots f_k(x+kt) e^{2\pi i P(t)}\ \frac{dt}{t}
$$
where $P$ ranges over all polynomials $P: \R \to \R$ of degree bounded by some fixed $d$, although this may require a more sophisticated counting lemma than the one given in \cite{gt-reg}.
\end{remark}

\end{document}